\documentclass[leqno,oneside,english]{article}
\usepackage[T1]{fontenc}
\usepackage[latin1]{inputenc}
\usepackage{amsmath}
\usepackage{amssymb}
\usepackage{color}

\makeatletter
\usepackage[T1]{fontenc}
\usepackage[latin1]{inputenc}

\usepackage{graphicx}
\numberwithin{equation}{section}

\makeatletter
\newcommand{\eps}{\epsilon}
\newcommand{\R}{{\mathbb R}}
\newcommand{\C}{{\mathbb C}}
\newcommand{\N}{{\mathbb N}}
\newcommand{\caH}{{\mathcal H}}
\newcommand{\caA}{{\mathcal A}}

\def\be{\begin{equation}}
\def\ee{\end{equation}}
\def\beq{\begin{eqnarray}}
\def\eeq{\end{eqnarray}}
\def\beqs{\begin{eqnarray*}}
\def\eeqs{\end{eqnarray*}}
\def\ea{\end{array}}
\def\ea{\end{array}}
\def\bt{\begin{thm}}
\def\et{\end{thm}}
\def\br{\begin{rk}}
\def\er{\end{rk}}
\def\bc{\begin{cor}}
\def\ec{\end{cor}}
\def\bl{\begin{lem}}
\def\el{\end{lem}}

\newcommand{\rot}{\mathop{\rm curl}\nolimits}
\newcommand{\curl}{\mathop{\rm curl}\nolimits}

\newtheorem{thm}{Theorem}[section]

\newtheorem{cor}[thm]{Corollary}

\newtheorem{exo}[thm]{Example}

\newtheorem{lem}[thm]{Lemma}

\newtheorem{rk}[thm]{Remark}

\newenvironment{proof}[1][Proof]{\textbf{#1.} }{\ \rule{0.5em}{0.5em}}
\newcommand{\ddiv}{{\rm div}}

\newcommand{\bn}{{\bf n}}

\makeatother

\usepackage{babel}
\makeatother
\author{Serge Nicaise\footnote{Universit\'e Polytechnique Hauts-de-France, LAMAV,
FR CNRS 2956,
F-59313 - Valenciennes Cedex 9, France,
Serge.Nicaise@uphf.fr},
Cristina Pignotti\footnote{DISIM, Universit\`a degli Studi dell'Aquila,
Via Vetoio, Loc. Coppito, 67010 L'Aquila, Italy, pignotti@univaq.it}}

\begin{document}
\title{Asymptotic behavior of dispersive electromagnetic waves in bounded domains}

  \maketitle

\begin{abstract}
  We analyze the    stability of Maxwell equations  in   bounded domains taking into account    electric and magnetization effects. Well-posedness of the  model is obtained by means of semigroup theory.
A {\em passitivity} assumption guarantees the boundedness of the associated semigroup.
Further  the exponential or polynomial decay of the energy is proved
under suitable  sufficient conditions.
Finally, several illustrative examples are presented.
\end{abstract}



\section{Introduction}

In this paper we analyze the   stability of Maxwell equations  with a general class of dispersion law
in a bounded domain $\Omega$
  of $\R^3$   with a Lipschitz boundary
$\Gamma$. More precisely,  the  Maxwell equations  in $\Omega$ are given by
\begin{equation}\label{Max} \left\{
\begin{tabular}{lllll}
& $ D_t -\rot H= 0$ in
$Q:= \Omega \times (0,+\infty)$,\vspace{2mm}
\\
& $ B_t +\rot E=0 $ in $Q$,\vspace{2mm}
\end{tabular}
\right.
\end{equation}
where $E$ and $H$ are respectively the electric and magnetic fields, while $D$ and $B$ are respectively the electric and magnetic flux densities.
But in case of electric and magnetization effects, these last ones take the form
\begin{eqnarray}\label{defD}
D(x,t)=\varepsilon(x) E(x,t)+P(x,t),\\
B(x,t)=\mu(x) H(x,t)+M(x,t),\label{defB}
\end{eqnarray}
where $\varepsilon$ (resp. $\mu$) is the permittivity (resp. permeability) of the medium, while
$P$ (resp. $M$) is the retarded electric polarization (resp. magnetization)
that in most applications (see \cite{Kristensson},\cite[Chapter 11]{Sihlova}, \cite{Cassieretall:17}) are of integral form
\begin{eqnarray}\label{defP}
P(x,t)=\int_{0}^t \nu_E(t-s, x) E(x,s)\,ds,\\
M(x,t)=\int_{0}^t \nu_H(t-s, x) H(x,s)\,ds,\label{defM}
\end{eqnarray}
where $\nu_E(t, x)$ (resp. $\nu_H(t, x)$) is the electric (resp. magnetic) susceptibility kernel. Some particular models (corresponding to particular kernels), like Debye, Lorentz or Drude models, can be reduced to a system coupling Maxwell's equations to a finite number of differential equations, see \cite{nic:scl2012}. In such a case semigroup theory can be applied to obtain existence and decay behavior of the solutions. Our goal is to analyze the general system
\eqref{Max} supplemented with the electric boundary conditions
\be\label{electricbc}
E\times \bn=0, \ H\cdot \bn=0 \hbox{ on } \Gamma=\partial \Omega,
\ee
 and initial conditions
 \be\label{IC}
E(\cdot, 0)=E_0, \ H(\cdot, 0)=H_0 \hbox{ in }  \Omega,
\ee
 and find sufficient conditions that guarantee exponential or polynomial decay (at infinity) of the solutions.

In \cite{Ioannidisetall_12} existence and uniqueness of solutions for problem  \eqref{Max}-\eqref{IC} are studied by transforming the system in a Volterra integral equation. We use here a different approach based on semigroup theory as in \cite{ContiGattiPata08,DaneseGeredeliPata15,GiorgiNasoPata:05,Munoz-Naso-Vuk:04}.

In
\cite{Becacheetall:18} the propagation of waves in unbounded dispersive media is studied by using the so-called Perfectly Matched Layers technique in order to realise artificial absorbing conditions. For dispersive isotropic Maxwell equations necessary and sufficient conditions for stability of the PML are given.

In this paper we restrict ourselves to the case when  the permittivity and the permeability
are positive constants, while   the kernels are real valued and do not depend on the space variable, namely we assume that  $\nu_E(t, x)=\nu_E(t)$ and $\nu_H(t, x)=\nu_H(t)$, this already  corresponds to a large class of physical examples, see for instance \cite{Sihlova,Kristensson}. We further assume that
$\nu_E,\nu_H \in K$, where
$K$ is the  set of kernels $\nu\in C^2([0,\infty))$,
that  satisfy
\be\label{asymptoicbehavior of derivatives}
\lim_{t\to \infty} \nu'(t)=0,
\ee
  and that there exists two positive constants
$C$ and $\delta$ (depending on $\nu$) such that
\be\label{hyponu''}
|\nu''(t)|\leq C e^{-\delta t},  \forall t\geq 0.
\ee

Again these assumptions cover a large class of physical models, see section \ref{sexamples} for some illustrative examples.

For shortness, we define the function $w$ by
\[
w(t)=Ce^{-\delta t}, \forall t\geq 0.
\]

The paper is organized as follows. In section 2 we study the well-posedness of the  model in an appropriate Hilbert setting by means of semigroup theory.
In section 3 we show, under the {\em passitivity} assumption (see \eqref{passitivityassump} below), that the semigroup associated to the model is bounded.
Sections 4 and 5 are devoted to the exponential or polynomial decay of the energy
under appropriate  sufficient conditions.
Finally, in section 6 we give several illustrative examples.

Let us finish this introduction with some notation used in
the paper:

\noindent The $L^2(\Omega)$-inner product (resp. norm)
will be denoted by  $(\cdot,\cdot)$ (resp. $\|\cdot\|$).   The usual
norm and semi-norm of $H^{s}(\Omega)$ ($s\geq 0$) are denoted by
$\|\cdot\|_{s,\Omega}$ and $|\cdot|_{s,\Omega}$, respectively.
For $s=0$ we drop the index $s$. By $a\lesssim b$, we mean that there exists a constant $C>0$ independent of $a$, $b$ and the time $t$, such that  $a\leq Cb$.

\section{Well-posedness result}

Even if existence result for problem \eqref{Max}-\eqref{IC} can be obtained using  Volterra integral equation method (see for instance \cite{Ioannidisetall_12}), we here prefer to use a first order past history framework (see \cite{Munoz-Naso-Vuk:04,GiorgiNasoPata:05,ContiGattiPata08} for second order framework and \cite{DaneseGeredeliPata15}
for first order one)
in order to formulate our system into a semigroup context (useful for the stability analysis).
First we notice that using the expressions
\eqref{defD} to \eqref{defM} into \eqref{Max}, we obtain the integro-differential system

\begin{equation}\label{Maxintdif} \left\{
\begin{tabular}{lllll}
& $ \varepsilon E_t+\nu_E(0) E+\int_{0}^t \nu'_E(t-s) E(\cdot, s)\,ds-\rot H= 0$ in
$Q$,\vspace{2mm}
\\
& $ \mu H_t+\nu_H(0) H+\int_{0}^t \nu'_H(t-s) H(\cdot, s)\,ds +\rot E=0 $ in $Q$.\vspace{2mm}
\end{tabular}
\right.
\end{equation}

Assuming for the moment that the solution $(E,H)$ of \eqref{Maxintdif} with boundary conditions \eqref{electricbc}
 and initial conditions \eqref{IC} exists, then   for all $(t, s)\in [0,\infty)\times (0,\infty)$  we introduce the summed past histories
\beq\label{eq:etaE}
\eta_E^t(\cdot, s)=\int_0^{\min\{s,t\}}E(\cdot, t-y)\,dy,
\\
\eta_H^t(\cdot, s)=\int_0^{\min\{s,t\}}H(\cdot, t-y)\,dy,\label{eq:etaH}
\eeq
that respectively satisfy the transport equation
\beq\label{transpE}
\partial_t\eta_E^t(\cdot, s)=-\partial_s\eta_E^t(\cdot, s)+E(\cdot, t),
\\
\label{transpH}
\partial_t\eta_H^t(\cdot, s)=-\partial_s\eta_H^t(\cdot, s)+H(\cdot, t),
\eeq
the boundary condition
\be\label{bceta}
\lim_{s\to 0}\eta_E^t(\cdot, s)=\lim_{s\to 0}\eta_H^t(\cdot, s)=0,
\ee
and the initial condition
\be\label{iceta}
\eta_E^0(\cdot, s)= \eta_H^0(\cdot, s)=0.
\ee
Since formal integration by parts
yields the identities
\beqs
\int_0^t \nu'_E(t-s) E(\cdot, s)\,ds=-\int_0^\infty \nu''_E(s) \eta^t_E(\cdot, s)\,ds,
\\
\int_0^t \nu'_H(t-s) H(\cdot, s)\,ds=-\int_0^\infty \nu''_H(s) \eta^t_H(\cdot, s)\,ds,
\eeqs
system \eqref{Maxintdif}  is (formally) equivalent to
\begin{equation}\label{Maxintdifequiv} \left\{
\begin{tabular}{lllll}
& $ \varepsilon E_t+\nu_E(0) E-\int_{0}^\infty \nu''_E(s) \eta^t_E(\cdot, s)\,ds-\rot H= 0$ in
$Q$,\vspace{2mm}
\\
& $ \mu H_t+\nu_H(0) H-\int_{0}^\infty  \nu''_H(s) \eta^t_H((\cdot, s)\,ds +\rot E=0 $ in $Q$,\vspace{2mm}
\end{tabular}
\right.
\end{equation}

All together by setting
\[
U=\left(
\begin{array}{llll}
E\\
H\\
\eta_E\\
\eta_H
\end{array}
\right),
\]
we obtain the Cauchy problem \begin{equation}\label{pbfirstorder}
\left\{\begin{array}{ll}U_t=\mathcal{A} U,\\U(0)=U_0,\end{array}\right.
\end{equation}
where
\be\label{defA}
\mathcal{A}\left(
\begin{array}{lll}
E\\
H\\
\eta_E\\
\eta_H
\end{array}
\right)=\left(
\begin{array}{lll}
\varepsilon^{-1}(-\nu_E(0) E+\int_{0}^\infty  \nu''_E(s) \eta_E(\cdot, s)\,ds+\rot H)\\
\mu^{-1}(-\nu_H(0) H+\int_{0}^\infty  \nu''_H(s) \eta_H(\cdot, s)\,ds -\rot E)\\
-\partial_s\eta_E(\cdot, s)+E\\
-\partial_s\eta_H(\cdot, s)+H
\end{array}
\right),
\ee
and
\[
U_0=(E_0, H_0,0,0)^\top.\]

The existence of a solution to (\ref{pbfirstorder}) is obtained by using semigroup theory    in the appropriate Hilbert setting
described here below:
First we introduce the Hilbert spaces
\beqs\label{2.1}  J(\Omega)&&=\{\chi \in L^2(\Omega
)^3\vert
\ddiv  \chi =0\},\\
\label{2.2}  \hat J(\Omega )&&=\{\chi \in J(\Omega)\vert
 \chi \cdot \nu =0 \hbox{ on } \Gamma \},\eeqs
 and recall that for an Hilbert space $X$ with inner product $(\cdot,\cdot)_X$
 and induced norm $\|\cdot\|_X$,  $L^2_w((0,\infty); X)$
 is the Hilbert space made of functions $\eta$ defined on $(0,\infty)$ with values in $X$
 such that
 \[
 \int_0^\infty \|\eta(s)\|_X^2 w(s)\,ds<\infty,
 \]
 with the natural inner product
  \[
 \int_0^\infty (\eta(s), \eta'(s))_X w(s)\,ds, \ \forall \eta,\eta'\in L^2_w((0,\infty); X).
 \]

Now we introduce the Hilbert space
$$
\caH=J(\Omega)\times \hat J(\Omega)\times L^2_w((0,\infty);J(\Omega))\times  L^2_w((0,\infty);\hat J(\Omega)),
$$
with the inner product
\beqs
&&((E,H,\eta_E, \eta_H)^\top,(E',H',\eta_E',\eta_H')^\top)_\caH:=\int_\Omega (\varepsilon E\cdot \bar E'+\mu H\cdot \bar H')\,dx
\\&&\hspace{1cm}+
\int_0^\infty \int_\Omega (\eta_E(x,s)\cdot \bar \eta'_E(x,s)+\eta_H(x,s)\cdot \bar \eta'_H(x,s))\,dx\,
 w(s)\,ds,
\eeqs
for all $(E,H,\eta_E, \eta_H)^\top,(E',H',\eta_E',\eta_H')^\top\in \caH.$

We then define the operator $\mathcal{A}$ as follows:
\beq\label{defDA}
&&\quad \mathcal{D}(\mathcal{A})=\left\{(E,H,\eta_E, \eta_H)^\top\!\in \mathcal{H}\vert
\rot E, \rot H\! \in L^2(\Omega
)^3, \, E\times n= 0 \!\hbox{ on } \Gamma,\right.\\
\nonumber
&&\hspace{2cm}
\partial_s \eta_E \in L^2_w((0,\infty);J(\Omega)), \partial_s \eta_H\in L^2_w((0,\infty);\hat J(\Omega))
\\
\nonumber
&&\left.\hspace{2cm}
\hbox{ and } \eta_E(0)=\eta_H(0)=0
\right\},
\eeq
and for all $U=(E,H,\eta_E, \eta_H)^\top\in \mathcal{D}(\mathcal{A})$,   $\mathcal{A}U$ is given by
(\ref{defA}).

We now check that $\mathcal{A}$ generates a $C_0$-semigroup   on $\caH$.
\bt\label{texistence}
The operator $\mathcal{A}$ defined by   \eqref{defA} with domain  \eqref{defDA}  generates a $C_0$-semigroup
 $(T(t))_{t\geq 0}$ on $\caH$.
Therefore for all $U_0\in \caH$,
 problem \eqref{pbfirstorder} has
a weak  solution $U\in C([0,\infty), \caH)$ given by $U(t) =T(t) U_0$, for all $t\geq 0$. If moreover $ U_0 \in D(\mathcal{A}^k)$, with $k\in \N^*$,   problem \eqref{pbfirstorder} has
a strong   solution $U\in C([0,\infty),D(\mathcal{A}^k))\cap C^1([0,\infty),
D(\mathcal{A}^{k-1}))$.
\et
\begin{proof}
It suffices to show that $\mathcal{A}-\kappa I$ is a maximal dissipative operator
for some $\kappa\geq 0;$
then by Lumer-Phillips' theorem it generates a
$C_0$-semigroup of contractions   on $\caH$
and consequently $\mathcal{A}$ generates a $C_0$-semigroup  on $\caH$.

Let us first show the  dissipativeness. Let $U=(E,H,\eta_E, \eta_H)^\top\in D(\mathcal{A})$ be fixed. Then by the definition of $\mathcal{A}$, we have
\beqs
&&(\mathcal{A}U, U)_\caH=
\int_\Omega \left ((-\nu_E(0) E+\int_{0}^\infty  \nu''_E(s)  \eta_E(\cdot, s)\,ds+\rot H) \cdot  \bar E
\right.
\\
&&\hspace{1cm}+
\left.
(-\nu_H(0) H+\int_{0}^\infty  \nu''_H(s) \eta_H(\cdot, s)\,ds -\rot E)  \cdot \bar H\right ) \,dx
  \\
&&\hspace{1cm}+\int_0^\infty \int_{\Omega} \left((-\partial_s\eta_E(\cdot, s)+E)\cdot \bar \eta_E
+(-\partial_s\eta_H(\cdot, s)+H)\cdot \bar \eta_H\right)
\,dx \,w(s) \,ds.
\eeqs
Note that by the density of $\mathcal{D}(\Omega)$ into
$\{E\in H(\rot, \Omega): E\times \bn=0$ on $\Gamma\}$, the next Green's formula holds
\be\label{green}
\int_\Omega \left (\rot  E \cdot  \bar H
- E\cdot \rot \bar H\right ) \,dx=0;
\ee
furthermore by integration by parts we have
\[
\begin{array}{l}
\displaystyle{\int_0^\infty \int_{\Omega}  \partial_s\eta_E(\cdot, s)\cdot \bar \eta_E\,dx
w(s) \,ds=}\\ \hspace{1 cm} \displaystyle{
-\int_0^\infty \int_{\Omega}  \eta_E(\cdot, s)\cdot   \partial_s\bar \eta_E\,dx
-\int_0^\infty \int_{\Omega}  |\eta_E(\cdot, s)|^2\,dx
w'(s) \,ds.}
\end{array}
\]
Using these identities we find
\beqs
\Re (\mathcal{A}U, U)_\caH
&=&-\int_\Omega  (\nu_E(0) |E|^2+\nu_H(0) |H|^2)\,dx
\\
&+&
\Re\int_\Omega \left (\int_{0}^\infty  \nu''_E(s)  \eta_E(\cdot, s)\,ds \cdot  \bar E
+\int_{0}^\infty  \nu''_H(s) \eta_H(\cdot, s)\,ds \cdot \bar H\right ) \,dx
  \\
&+&\frac{1}{2}\int_0^\infty \int_{\Omega}  (|\eta_E(\cdot, s)|^2+(|\eta_H(\cdot, s)|^2) \,dx
w'(s) \,ds
\\
&+&\Re \int_0^\infty \int_{\Omega} \left(E \cdot \bar \eta_E+H\cdot \bar \eta_H\right)
\,dx \,w(s) \,ds.
\eeqs
As $w'(s) \leq 0$, we deduce that
\beqs
\Re (\mathcal{A}U, U)_\caH
&\leq &-\int_\Omega  (\nu_E(0) |E|^2+\nu_H(0) |H|^2)\,dx
\\
&+&
\Re\int_\Omega \left (\int_{0}^\infty  \nu''_E(s)  \eta_E(\cdot, s)\,ds \cdot  \bar E
+\int_{0}^\infty  \nu''_H(s) \eta_H(\cdot, s)\,ds \cdot \bar H\right ) \,dx
\\
&+&\Re \int_0^\infty \int_{\Omega} \left(E \cdot \bar \eta_E+H\cdot \bar \eta_H\right)
\,dx \,w(s) \,ds.
\eeqs

Now using the assumption \eqref{hyponu''}, the definition of $w$ and Cauchy-Scharz's inequality, we find
that there exists a positive constant $\kappa$ such that
\beqs\label{Adissip}
\Re (\mathcal{A}U, U)_\caH
&\leq \kappa & \|U\|_\caH^2.
\eeqs
This shows that  $\mathcal{A}-\kappa I$ is   dissipative.

Let us go on with the maximality. Let $\lambda>0$ be fixed.
For $(F,G,R,S)^\top\in \caH$, we look for $U=(E,H,\eta_E, \eta_H)^\top\in D(\mathcal{A})$ such that
\be\label{max}
(\lambda I-\caA) U=(F,G,R,S)^\top.
\ee
According to  (\ref{defA}) this is equivalent to
\beq
\eps\lambda E+\nu_E(0) E-\int_{0}^\infty  \nu''_E(s) \eta_E(\cdot, s)\,ds-\rot H=\eps F,\label{max1}\\
\mu \lambda H+\nu_H(0) E-\int_{0}^\infty  \nu''_H(s) \eta_H(\cdot, s)\,ds+\rot E= \mu G,\label{max2}\\
\lambda \eta_E+\partial_s\eta_E(\cdot, s)-E=R,\label{max3}\\
\lambda \eta_H+\partial_s\eta_H(\cdot, s)-H=S.\label{max4}
\eeq

Assume for the moment that $U$ exists.
Then the two last   equations allow to eliminate $\eta_E$ and $\eta_H$
since they are equivalent to
\beq
 \eta_E(s)=\frac{1-e^{-\lambda s}}{\lambda} E+\int_0^se^{-\lambda (s-y)} R(y)\,dy,\label{max3equiv}\\
  \eta_H(s)=\frac{1-e^{-\lambda s}}{\lambda} H+\int_0^se^{-\lambda (s-y)} S(y)\,dy.\label{max4equiv}
\eeq
Thus  inserting these expressions
in \eqref{max1} and  \eqref{max2}, we find that
\beq
&&\eps\lambda E+\left(\nu_E(0) -\frac{1}{\lambda} \int_{0}^\infty  \nu''_E(s) (1-e^{-\lambda s})
\,ds\right) E-\rot H=\eps F+r(\lambda),\label{max1bis}\\
&&\mu \lambda H+\left(\nu_H(0) -\frac{1}{\lambda} \int_{0}^\infty  \nu''_H(s) (1-e^{-\lambda s})
\,ds\right) H+\rot E= \mu G+s(\lambda),\label{max2bis}
\eeq
where
\beq\label{defr}
r(\lambda)&=&\int_{0}^\infty \nu''_E(s) \int_0^se^{-\lambda (s-y)} R(y)\,dy\,ds,\\
s(\lambda)&=&\int_{0}^\infty \nu''_H(s) \int_0^se^{-\lambda (s-y)} S(y)\,dy\,ds,
\label{defs}
\eeq
that have the regularity $r\in J(\Omega)$ and $s\in \hat J(\Omega)$.
But two integrations by parts allow to show that (see section \ref{appendix})
\[
\nu_E(0) -\frac{1}{\lambda} \int_{0}^\infty  \nu''_E(s) (1-e^{-\lambda s})
\,ds=\nu_E(0)+\frac{1}{\lambda}(\nu_E'(0)+\mathcal{L} \nu''_E(\lambda))=
\lambda \mathcal{L} \nu_E(\lambda),
\]
where we recall that
 $\mathcal{L} \nu_E$ is the Laplace transform of $\nu_E$, see \eqref{lapTransform} below.
Hence the previous identities
 \eqref{max1bis} and  \eqref{max2bis}  may be equivalently written as
\beq
\lambda \left(\eps + \mathcal{L} \nu_E(\lambda) \right) E-\rot H=\eps F+r(\lambda),\label{max1ter}\\
\lambda \left( \mu + \mathcal{L} \nu_H(\lambda)\right) H+\rot E= \mu G+s(\lambda).\label{max2ter}
\eeq
Owing to \eqref{lapT7},
for $\lambda$ large enough, we will have
\[
\eps + \mathcal{L} \nu_E(\lambda) >0  \hbox{ as well as  } \mu + \mathcal{L} \nu_H(\lambda)>0.
\]
Therefore for  $\lambda$ large enough the system \eqref{max1ter}-\eqref{max2ter} enters in a standard framework (see for instance
\cite[Lemma 3.1]{nic:00}) and a unique solution $(E, H)$ exists with the regularity
\beqs
E\in X_N(\Omega)=\{U\in J(\Omega): \curl U\in L^2(\Omega)^3 \hbox{ and } U\times n=0 \hbox{ on } \Gamma\},
\\
H\in X_T(\Omega)=\{U\in \hat J(\Omega): \curl U\in L^2(\Omega)^3 \},
\eeqs
because $\eps F+r$ (resp. $\mu G+s$) belongs to $J(\Omega)$ (resp. $\hat J(\Omega)$).

The surjectivity of $\lambda I-\caA$ for $\lambda$ large enough  finally holds because
once $E$ and $H$ are given, we obtain $\eta_E$ and $\eta_H$ with the help of
\eqref{max3equiv} and \eqref{max4equiv} respectively
and easily check their right requested regularity.
\end{proof}

\section{Boundedness of the semigroup}

In order to apply standard results on the decay of semigroups (see Lemmas \ref{ab}, \ref{lemraoexp}
and  \ref{lemrao} below),  the first step is to show that the semigroup
 $(T(t))_{t\geq 0}$ generated by
 $\mathcal{A}$ is bounded. This property is based on the passitivity assumption (or equivalently the assumption that the material is passive, see \cite[Definition 2.5]{Cassieretall:17} and \cite[(2.15)]{Nguyen-Vinoles:18}), that says that
 (see \eqref{LapT10})
 \be\label{passitivityassump}
 \Re \left(i\omega {\mathcal L}\nu_E(i\omega)\right)\geq 0,\
  \Re \left(i\omega  {\mathcal L}\nu_H(i\omega)\right)\geq 0, \ \forall \omega\in \R.
 \ee
 Note that this property is equivalent to
 \be\label{passitivityassumpequiv}
 \omega  \Im  {\mathcal L}\nu_E(i\omega)\leq 0,\
  \omega  \Im     {\mathcal L}\nu_H(i\omega)\leq 0,\  \forall \omega\in \R.
 \ee

\begin{lem}\label{l:bdsg}  Under the additional assumption \eqref{passitivityassump}, there exists a positive constant $M$
such that
\be\label{bdsg}
\|T(t)\|\leq M, \ \forall t\geq 0.
\ee
\end{lem}
\begin{proof}
Take $U_0=(E_0,H_0, \eta_E^0, \eta_H^0)\in D(\caA)$
and let
$U(t)=(E(t), H(t), \eta^t_E, \eta^t_H)=T(t)U_0$, for all $t\geq 0$.
Then by Theorem \ref{texistence},
  $U\in C([0,\infty),D(\mathcal{A}))\cap C^1([0,\infty),
\caH)$ is
a strong   solution of problem (\ref{pbfirstorder}), which means that
\eqref{transpE}-\eqref{transpH}
and \eqref{Maxintdifequiv} hold for all $t>0$.

But we notice that
\beq\label{etaE}
\eta^t_E(\cdot, s)= \tilde \eta_E^0(\cdot,s-t)+\int_0^{\min\{s,t\}}E(\cdot, t-y)\,dy,\\
\label{etaH}
\eta^t_H(\cdot, s)= \tilde \eta_H^0(\cdot,s-t)+\int_0^{\min\{s,t\}}H(\cdot, t-y)\,dy,
\eeq
where
 $\tilde \eta_E^0$ is the extension of  $\eta_E^0$ by zero on $(-\infty, 0)$.
 Inserting these expressions in  \eqref{Maxintdifequiv}, we find that
 \begin{equation}\label{Maxintdifequivelemination eta} \left\{
\begin{tabular}{lllll}
& $ \varepsilon E_t+\nu_E(0) E+\int_{0}^t \nu'_E(t-s) E(\cdot, s)\,ds-\rot H=F(t)  $ in
$Q$,\vspace{2mm}
\\
& $ \mu H_t+\nu_H(0) H+\int_{0}^\infty  \nu'_H(s) H((\cdot, s)\,ds +\rot E=G(t)  $ in $Q$,\vspace{2mm}
\end{tabular}
\right.
\end{equation}
where
\beqs
F(t):=\int_{0}^\infty  \nu''_E(s) \tilde \eta^0_E(\cdot, s-t)\,ds=\int_{t}^\infty  \nu''_E(s)   \eta^0_E(\cdot, s-t)\,ds,
\\
G(t):=\int_{0}^\infty  \nu''_H(s) \tilde \eta^0_H(\cdot, s-t)\,ds=\int_{t}^\infty  \nu''_H(s)  \eta^0_H(\cdot, s-t)\,ds.
\eeqs
Now we remark that
\[
\nu_E(0) E+\int_{0}^t \nu'_E(t-s) E(\cdot, s)\,ds=\frac{d}{dt}\left(\int_{0}^t \nu_E(t-s) E(\cdot, s)\,ds\right),
\]
and therefore system \eqref{Maxintdifequivelemination eta} is equivalent to
\begin{equation}\label{Maxintdifequivelemination etaequiv} \left\{
\begin{tabular}{lllll}
& $ \varepsilon E_t+\frac{d}{dt}\left(\int_{0}^t \nu_E(t-s) E(\cdot, s)\,ds\right)-\rot H=F(t)  $ in
$Q$,\vspace{2mm}
\\
& $ \mu H_t+\frac{d}{dt}\left(\int_{0}^t \nu_H(t-s) H(\cdot, s)\,ds\right)+\rot E=G(t)  $ in $Q.$\vspace{2mm}
\end{tabular}
\right.
\end{equation}
Now we adapt an argument used in the proof of Theorem 3.1 from \cite{Nguyen-Vinoles:18}.
For a fixed $T>0$, if we multiply the first identity by $\bar E(t)$ and the second one by
$\bar H(t)$, integrate both in $\Omega\times (0,T)$ and take the sum, we find that
\beqs
&&\int_0^T\int_\Omega \left( \left(\varepsilon E_t+\frac{d}{dt}\left(\int_{0}^t \nu_E(t-s) E(\cdot, s)\,ds\right)-\rot H(t)\right)\cdot \bar E(t)\right.
\\\hspace{1cm}&&+
\left.
\left(\mu H_t+\frac{d}{dt}\left(\int_{0}^t \nu_H(t-s) H(\cdot, s)\,ds\right)+\rot E(t)\right)\cdot \bar H(t)
\right)dxdt
\\
\hspace{1cm}&&=\int_0^T\int_\Omega(F(t)\cdot \bar E(t)+ G(t)\cdot \bar  H(t))dxdt.
\eeqs
Taking the real part of this identity and applying  Green's formula  \eqref{green}, we get
\beqs
&&\Re\int_0^T\int_\Omega \left( \left(\varepsilon E_t+\frac{d}{dt}\left(\int_{0}^t \nu_E(t-s) E(\cdot, s)\,ds\right)\right)\cdot \bar E(t)\right.
\\\hspace{1cm}&&+
\left.
\left(\mu H_t+\frac{d}{dt}\left(\int_{0}^t \nu_H(t-s) H(\cdot, s)\,ds\right)\right)\cdot \bar H(t)
\right)dxdt
\\
\hspace{1cm}&&=\Re\int_0^T\int_\Omega(F(t)\cdot \bar E(t)+ G(t) \cdot \bar H(t))dxdt.
\eeqs
Now if we define
$\tilde E_T$ (and similarly for $\tilde H_T$) by
\[
\tilde E_T(\cdot, t)=
\left\{
\begin{tabular}{ll}
$ E(\cdot, t)$ &  if $t\in (0,T),$\\
 0 &   else,
\end{tabular}
\right.
\]
the previous identity can be written as
\beq\label{bdsg:1}
&&\Re\int_0^T\int_\Omega  \left(\varepsilon E_t\cdot  \bar E(t)+\mu H_t\cdot \bar H(t)
\right)dxdt
\\\hspace{1cm}&&
=-\Re\int_\R\int_\Omega \left(  \frac{d}{dt} (\tilde \nu_E\star_t \tilde E_T)(t)\cdot  \overline{\tilde E_T}(t)
+ \frac{d}{dt} (\tilde \nu_H\star_t \tilde H_T)(t)\cdot  \overline{\tilde H_T}(t)\right)dxdt
\nonumber
\\
\hspace{1cm}&&+\Re\int_0^T\int_\Omega(F(t)\cdot \bar E(t)+ G(t) \cdot \bar  H(t))dxdt,
\nonumber
\eeq
where $f \star_t g$ means the convolution in $\R$, namely
\[
(f \star_t g)(t)=\int_{\R} f(t-s) g(s)\,ds, \ \forall t\in \R.
\]
Now by Parseval's identity, we have
\beqs
&&\int_\R\int_\Omega \left(  \frac{d}{dt} (\tilde \nu_E\star_t \tilde E_T)(t)\cdot  \bar E_T(t)
+ \frac{d}{dt} (\tilde \nu_H\star_t \tilde H_T)(t)\cdot  \bar H_T(t)\right)dxdt
\\
\hspace{0.5cm}&&=
\int_\Omega \int_\R i\omega \left(\mathcal{F}((\tilde \nu_E)(i\omega) | (\mathcal{F}((\tilde E_T)(i\omega)|^2
+\mathcal{F}((\tilde \nu_H)(i\omega) | (\mathcal{F}((\tilde H_T)(i\omega)|^2\right)
d\omega dx.
\eeqs
By our passivity assumption   \eqref{passitivityassump}, we deduce that
\[
\Re
\int_\R\int_\Omega \left(  \frac{d}{dt} (\tilde \nu_E\star_t \tilde E_T)(t)\cdot  \bar E_T(t)
+ \frac{d}{dt} (\tilde \nu_H\star_t \tilde H_T)(t)\cdot  \bar H_T(t)\right)dxdt\geq 0.
\]
This estimate in the identity \eqref{bdsg:1} leads to
\beq\label{bdsg:2}
\int_\Omega  \left(\varepsilon |E(x, T)|^2  +\mu |H(x, T)|^2\right) dx
\leq\int_\Omega  \left(\varepsilon E_0(x)|^2  +\mu |H_0(x)|^2\right) dx
\\
+
 2 \Re\int_0^T\int_\Omega(F(t)\cdot \bar E(t)+ G(t)\cdot \bar H(t))dxdt.
\nonumber
\eeq
By setting
\[
\mathcal{E}(t)=\frac{1}{2}\int_\Omega  \left(\varepsilon |E(x, t)|^2  +\mu |H(x, t)|^2\right) dx, \forall t\geq 0,
\]
by using Cauchy-Schwarz's inequality in the last estimate, we obtain
\[
 \mathcal{E}(T)\leq \mathcal{E}(0)+\sqrt{2}\int_0^T \left(\int_\Omega(|F(x,t)|^2+|G(x,t)|^2dx \right)^{\frac{1}{2}} \mathcal{E}(t)^{\frac{1}{2}}  dt, \ \forall T>0.
 \]
By Gronwall's inequality (see for instance \cite[Lemma 3.1]{Nguyen-Vinoles:18}), we deduce that
\be
\label{bdsg:3}
 \mathcal{E}(t)\leq \left(\mathcal{E}(0)^{\frac{1}{2}}+\frac{\sqrt{2}}{2}\int_0^t \left(\int_\Omega(|F(x,s)|^2+|G(x,s)|^2)dx \right)^{\frac{1}{2}}  ds\right)^2, \ \forall t>0.
 \ee

 Now we need to estimate the term
 \[
 \int_0^t \left(\int_\Omega(|F(x,s)|^2+|G(x,s)|^2)dx \right)^{\frac{1}{2}}  ds.
 \]
 But using the definition of $F$ and $G$,  the assumption \eqref{hyponu''} and Cauchy-Schwarz's inequality, we see that
 \beqs
 && \int_0^t \left(\int_\Omega(|F(x,s)|^2+|G(x,s)|^2)dx \right)^{\frac{1}{2}}  ds
 \\&&\hspace{1cm} \lesssim  \left(\int_0^t e^{-\delta s} ds\right) (\|\eta^0_E\|_{L^2_w((0,\infty);J(\Omega))}
  +
  \|\eta^0_H\|_{ L^2_w((0,\infty);\hat J(\Omega))}
  ),
  \eeqs
  and therefore
   \[
  \int_0^t \left(\int_\Omega(|F(x,s)|^2+|G(x,s)|^2)dx \right)^{\frac{1}{2}}  ds
  \lesssim \|\eta^0_E\|_{L^2_w((0,\infty);J(\Omega))}+
  \|\eta^0_H\|_{ L^2_w((0,\infty);\hat J(\Omega))}.
  \]
  Using this estimate in \eqref{bdsg:3} we have obtained that
  \be
\label{bdsg:4}
 \mathcal{E}(t)\lesssim \|U_0\|_{\caH}^2.
 \ee

Now we come back to \eqref{etaE} and \eqref{etaH} to estimate the norm of $\eta^t_E$
and $\eta^t_H$. Let us perform the estimation for $\eta^t_E$.
By \eqref{etaE}, we have
\beqs
\| \eta^t_E(\cdot, s)\|_{L^2_w((0,\infty);J(\Omega))}^2
&\leq& 2\int_0^\infty w(s) \int_\Omega |\tilde \eta_E^0(x,s-t)|^2dx ds
\\
&+&2\int_0^\infty w(s) \int_\Omega \left|\int_0^{\min\{s,t\}}E(x, t-y)\,dy\right|^2 dx ds.
\eeqs
The first term is easily estimated because
a change of variable  and the property
$w(s+t)= e^{-\delta t} w(s)$, valid for all $s,t\geq 0,$ yield
\beqs
\int_0^\infty w(s) \int_\Omega |\tilde \eta_E^0(x,s-t)|^2dx ds
&=&\int_t^\infty w(s) \int_\Omega |\eta_E^0(x,s-t)|^2dx ds
\\
&=&\int_0^\infty w(s+t) \int_\Omega |\eta_E^0(x,s)|^2dx ds
\\
&\leq & e^{-\delta t}\int_0^\infty w(s) \int_\Omega |\eta_E^0(x,s)|^2dx ds.
\eeqs
This means that
\be\label{estimeetermetransport}
\|\tilde \eta_E^0(x,\cdot -t)\|_{L^2_w((0,\infty;  J(\Omega))}\leq e^{-\frac{\delta t}{2}}
\|\eta_E^0\|_{L^2_w((0,\infty;  J(\Omega))}.
\ee
For the second term, by Cauchy-Schwarz's inequality and Fubini's theorem, we have
\beqs
&&\int_0^\infty w(s) \int_\Omega \left|\int_0^{\min\{s,t\}}E(x, t-y)\,dy\right|^2 dx ds
\\
&&\hspace{1cm}\leq
\int_0^\infty w(s) \int_\Omega s   \int_0^{\min\{s,t\}}|E(x, t-y)|^2\,dy dx ds
\\
&&\hspace{1cm}\leq 2\int_0^\infty s w(s)     \int_0^{\min\{s,t\}}
\mathcal{E}(t-y)\,dy  ds.
\eeqs
Hence using the estimate \eqref{bdsg:4}, we find
\beqs
\int_0^\infty w(s) \int_\Omega \left|\int_0^{\min\{s,t\}}E(x, t-y)\,dy\right|^2 dx ds
\lesssim \|U_0\|_{\caH}^2 \int_0^\infty w(s)  s^2 ds
\lesssim \|U_0\|_{\caH}^2.
\eeqs

These last estimates show that
\beqs
\| \eta^t_E(\cdot, s)\|_{L^2_w((0,\infty);J(\Omega))}^2\lesssim \|U_0\|_{\caH}^2.
\eeqs
Since the same arguments yield
\beqs
\| \eta^t_H(\cdot, s)\|_{L^2_w((0,\infty);\hat J(\Omega))}^2\lesssim \|U_0\|_{\caH}^2,
\eeqs
the combination of these two estimates with
\eqref{bdsg:4} leads to
\[
\|U(t)\|_\caH^2\lesssim \|U_0\|_{\caH}^2, \ \forall t>0.
\]
Since $D(\mathcal{A})$ is dense in $\caH$, we conclude that
\[
\|T(t)U_0\|_\caH^2\lesssim \|U_0\|_{\caH}^2, \ \forall t>0,\  U_0\in \caH,
\]
which is the claim.
\end{proof}

\begin{cor}\label{corol:bdsg}
Under the additional assumption \eqref{passitivityassump},  then the resolvent set $ \rho(\mathcal{A})$
of $\mathcal{A}$ contains the right-half plane, namely
\[
\{\lambda\in \mathbb{C}: \Re \lambda>0\} \subset \rho(\mathcal{A}).
\]
\end{cor}
\begin{proof}
Direct consequence of Lemma \ref{l:bdsg}  and of Theorem 5.2.1 of  \cite{MR1886588}.
\end{proof}

\section{Strong stability\label{strongstab}}

One simple way to prove the strong stability of (\ref{pbfirstorder})
 is to use the following
 theorem due to
 Arendt--Batty and Lyubich--V\~u     (see  \cite{arendt:88,Lyubich-Vu:88}).

\bt[Arendt--Batty/Lyubich--V\~u]\label{ab}Let $X$ be a reflexive Banach space and
 $(T(t))_{t\ge 0}$ be a  bounded  $C_0$ semigroup generated by $A$ on $X$.
 Assume that $(T(t))_{t\ge 0}$ is bounded and that no eigenvalues of $A$ lie on the imaginary axis. If $\sigma(A)\cap i\R$ is countable, then $(T(t))_{t\ge 0}$ is stable.
\et

We now want to take advantage of this Theorem. Since the resolvent of our operator  is not compact,
we have to analyze the full spectrum of $\mathcal A$ on the imaginary axis. For that purpose, we actually need a
stronger assumption than the passivity, namely in addition to \eqref{passitivityassump},
we need that

\be\label{passitivityassump+}
 \Re \left(i\omega {\mathcal L}\nu_E(i\omega)\right)+ \Re \left(i\omega  {\mathcal L}\nu_H(i\omega)\right)> 0, \ \forall \omega\in \R^*=\R\setminus\{0\}.
 \ee
 As before  this property is equivalent to
 \be\label{passitivityassumpequiv+}
 \omega  \Im  {\mathcal L}\nu_E(i\omega)+
  \omega  \Im     {\mathcal L}\nu_H(i\omega)< 0, \ \forall \omega\in \R^*.
 \ee

 We first prove a preliminary result related to a family of operators defined in $\caH_1:=J(\Omega)\times \hat J(\Omega)$. First let us consider the unbounded operator
$\mathcal{B}$ from  $\caH_1:=J(\Omega)\times \hat J(\Omega)$ into itself
with domain
\[
D(\mathcal{B}):=\{(E,H)\in \caH_1\vert
\rot E, \rot H \in L^2(\Omega
)^3, \hbox{ and } E\times n= 0  \hbox{ on } \Gamma\},
\]
defined by
\[
\mathcal{B}((E,H))=   (\eps E-\rot H, \mu H+\rot E).
\]
As said before $\mathcal{B}$
is an isomorphism from $D(\mathcal{B})$  into $\caH_1$, with a compact resolvent.
Consequently for any $\omega\in \R$, the operator
\[
\mathcal{B}_\omega((E,H))=   (i\omega \left(\eps + \mathcal{L} \nu_E(i\omega) \right) E-\rot H,
i\omega \left( \mu + \mathcal{L} \nu_H(i\omega)\right) H+\rot E),
\]
with the same domain as $\mathcal{B}$
is a compact perturbation of $\mathcal{B}$. Hence for all $\omega$,
$\mathcal{B}_\omega$ is a Fredholm operator of index 0. Hence it will be an isomorphism if and only if it is injective.
This is proved in the next lemma.
\bl\label{liso}
Under the additional assumptions \eqref{passitivityassump} and \eqref{passitivityassump+},
and if $\Omega$ is simply connected and its boundary connected, then
the operator $\mathcal{B}_\omega$ is   an isomorphism from $D(\mathcal{B})$ into $\caH_1$.
\el
\begin{proof}
Let $(E,H)\in \ker \mathcal{B}_\omega$, then we have
$$
\begin{array}{l}
\displaystyle{
0=(\mathcal{B}_\omega(E,H), (E,H))_{\caH_1}}
\\
\ \ =\displaystyle{
\int_\Omega ((i\omega \left(\eps + \mathcal{L} \nu_E(i\omega) \right) E-\rot H)\cdot \bar E
}\\
\hspace{2,5 cm} \displaystyle{+(i\omega \left( \mu + \mathcal{L} \nu_H(i\omega)\right) H+\rot E)\cdot \bar H)\,dx.}
\end{array}
$$
Hence applying Green's formula and taking the real part of the identity,
we find that
\beqs
0=
\int_\Omega (\Re(i\omega \left(\eps + \mathcal{L} \nu_E(i\omega) \right) |E|^2
+\Re(i\omega \left( \mu + \mathcal{L} \nu_H(i\omega)\right) H)|H|^2)\,dx.
\eeqs
By our  assumptions \eqref{passitivityassump} and \eqref{passitivityassump+}, we may distinguish three cases:
\\
1. If $\Re(i\omega \left(\eps + \mathcal{L} \nu_E(i\omega) \right)>0$, we deduce that
$E=0$
and by the definition of $ \mathcal{B}_\omega$, we deduce that
$\rot H=0$. This property added to the fact that $H\in \hat J(\Omega)$
allows to conclude that $H=0$ owing to Proposition 3.14 of \cite{AmroucheBernardiDaugeGirault98}.
\\
2. If $\Re(i\omega \left(\eps + \mathcal{L} \nu_H(i\omega) \right)>0$, we deduce that
$H=0$
and by the definition of $ \mathcal{B}_\omega$, we deduce that
$\rot E=0$. This property added to the fact that $E$ is divergence free and
satisfies
\[
E\times n= 0  \hbox{ on } \Gamma,
\]
allows to conclude that $E=0$ owing to Proposition 3.18 of \cite{AmroucheBernardiDaugeGirault98}.
\\
3. If $\Re(i\omega \left(\mathcal{L} \nu_E(i\omega)+ \mathcal{L} \nu_H(i\omega) \right)=0$ at $\omega=0$, then we directly deduce that $\rot E=0=\rot H=0$, and we conclude that $E=H=0$ with the help of
Propositions 3.14 and 3.18 of \cite{AmroucheBernardiDaugeGirault98}.
\end{proof}

 \br {\rm
 Obviously the assumption that  $\Omega$ is simply connected and that its boundary is connected
 can be weakened if \eqref{passitivityassump+} can be replaced by a stronger assumption.
 }
 \er

\bl\label{lresolvante}
Under the  assumptions of Lemma \ref{liso},
\[
i \mathbb{R}\equiv \bigr\{i \beta \bigm|\beta \in \mathbb{R}
\bigr\} \subset
\rho ({\mathcal A}).
\]
\el
\begin{proof}
The proof is similar to the proof of the maximality of $\mathcal{A}$.
Indeed fix $\omega\in \R$ and let
  $(F,G,R,S)^\top\in \caH$. Then we look for $U=(E,H,\eta_E, \eta_H)^\top\in D(\mathcal{A})$ such that
\be\label{maxiomega}
(i\omega I-\caA) U=(F,G,R,S)^\top.
\ee
Arguing as in the proof of the maximality, this means that we first look for $(E,H)$  solution of
\eqref{max1ter}-\eqref{max2ter} with $\lambda=i\omega$, or
equivalently solution of
\be\label{maxiomegaEH}
\mathcal{B}_\omega (E,H)=(\eps F+r(i \omega),\mu G+s(i \omega)).
\ee
Note that
$r(i\omega)$ (resp. $s(i\omega)$) belongs to $J(\Omega)$ (resp.
$\hat J(\Omega)$) because by Fubini's theorem and Cauchy-Schwarz's inequality we have
\beqs
\|r(i\omega)\|_\Omega
&\leq&\int_{0}^\infty |\nu''_E(s)| \int_0^s \|R(\cdot,y)\|_\Omega\,dy\,ds
\\
&\lesssim &\int_{0}^\infty w(s) \int_0^s \|R(\cdot,y)\|_\Omega\,dy\,ds
\\
&\lesssim &\int_{0}^\infty \|R(\cdot,y)\|_\Omega    w(y)  \,dy\\
&\lesssim & (\int_{0}^\infty   w(y)  \,dy)^{\frac{1}{2}} \|R\|_{L^2_w(0,\infty); J(\Omega))}.
 \eeqs
By Lemma \ref{liso}, there exists a unique solution $ (E,H)\in D(\mathcal{B})$ to \eqref{maxiomegaEH}.
As before, we obtain $\eta_E$ and $\eta_H$ with the help of
\eqref{max3equiv} and \eqref{max4equiv} respectively (with $\lambda=i\omega$)
and easily check their right requested regularity.
\end{proof}

As a direct consequence of this Lemma and Theorem \ref{ab}, we obtain the following result.

\bl\label{lstrongstab}
Under the  assumptions of Lemma \ref{liso},
 $(T(t))_{t\ge 0}$ is stable, in other words
 \[
 T(t)U_0\to 0 \hbox{ in } \caH, \hbox{ as } t\to \infty, \ \forall U_0\in \caH.
 \]
 In particular the solution $(E(t), H(t))$ of \eqref{Maxintdif}, \eqref{electricbc}
 and  \eqref{IC} satisfies
 \[
 \|E(t)\|_\Omega+\| H(t)\|_\Omega \to 0   \hbox{ as } t\to \infty, \ \forall (E_0, H_0)\in \caH_1.
 \]
\el

\section{Stability results \label{sststab}}

Our stability results are based on
a frequency domain approach, namely for  the exponential  decay of the
energy   we use the following result (see
 \cite{pruss:84} or  \cite{huang:85}):
\begin{lem}\label{lemraoexp}
Let $(e^{t{\mathcal L}})_{t\ge 0}$  be a bounded  $C_0$ semigroup  on a Hilbert
space $H$. Then it is   exponentially stable, i.e., it satisfies
$$||e^{t{\mathcal L}}U_0|| \leq C \, e^{-\omega t} ||U_0||_{H},\quad
 \forall U_0\in H,\quad\forall t\geq 0,$$
for some positive constants $C$ and $\omega$  if and only if
\begin{equation}
i \mathbb{R} \subset
\rho ({\mathcal L}),
 \label{1.8w} \end{equation} and
\begin{equation} \sup_{\beta \in \R}   \,
\|(i\beta -{\mathcal L})^{-1}\| <\infty. \label{1.9l=0}
\end{equation}
\end{lem}

On the contrary the polynomial decay of the energy  is based on  the following result
stated in Theorem 2.4 of \cite{borichev:10} (see also
\cite{batkai:06,Liu_Rao_05} for weaker variants).
\begin{lem}\label{lemrao}
Let $(e^{t{\mathcal L}})_{t\ge 0}$  be a bounded  $C_0$ semigroup  on a Hilbert
space $H$ such that its generator $\mathcal L$ satisfies \eqref{1.8w} and let $\ell$ be a  fixed positive real number.
Then the following properties are equivalent
\beq
\nonumber
&&||e^{t{\mathcal L}}U_0|| \leq C \, {t^{-\frac{1}{\ell}}} ||U_0||_{{\mathcal D}({\mathcal L})},\quad
 \forall U_0\in\mathcal{D}(\mathcal{L}),\quad\forall t>1,\\
&&||e^{t{\mathcal L}}U_0|| \leq C \, t^{-1} ||U_0||_{{\mathcal
D}({\mathcal L}^\ell)},\quad
 \forall U_0\in\mathcal{D}(\mathcal{L}^\ell),\quad\forall t>1,
 \nonumber
 \\
  &&\sup_{\beta\in \R} \frac{1}{1+|\beta|^\ell} \,
\|(i\beta -{\mathcal L})^{-1}\| <\infty. \label{1.9}
\eeq
\end{lem}

As Lemma \ref{lresolvante} guarantees that the assumption (\ref{1.8w}) holds, it remains to check whether
\eqref{1.9l=0} or \eqref{1.9} is valid. This is possible by improving the assumption \eqref{passitivityassump+}
with a precise behavior of $ \Re \left(i\omega {\mathcal L}\nu_E(i\omega)\right)$ and of $\Re \left(i\omega  {\mathcal L}\nu_H(i\omega)\right)$ at infinity. More precisely, we suppose that there exist four non negative constants $\sigma_E$,
$\sigma_H$, $\omega_0$, and $m$ with $\sigma_E+\sigma_H>0$
such that
\beq\label{passitivityassump++}
&& \Re \left(i\omega {\mathcal L}\nu_E(i\omega)\right) |X|^2+ \Re \left(i\omega  {\mathcal L}\nu_H(i\omega)\right)|Y|^2
 \\
&&\hspace{1,5 cm} \geq |\omega|^{-m} (\sigma_E|X|^2+\sigma_H|Y|^2),\  \forall X,Y\in \C^3, \ \omega\in \R: |\omega|\geq \omega_0.
 \nonumber
 \eeq

 \bl\label{lasbehavior resolvent}
In addition to the assumptions of Lemma   \ref{liso}, assume that  \eqref{passitivityassump++} holds. Then
the operator $\mathcal{A}$ satisfies \eqref{1.9} with $\ell=m$.
\el
\begin{proof}
We  use a contradiction argument, namely suppose that
(\ref{1.9})  is false. Then there exist a sequence of
real numbers $\beta_n\rightarrow+\infty$ and a sequence of vectors
$z_n=(E_n,H_n,\eta_{E,n}, \eta_{H,n})^\top$ in $\mathcal{D}(\mathcal{A})$   with
\be
\label{bdsequence}
\left\|z_n\right\|_{\mathcal{H}}=1,
\ee satisfying
\beq
\label{ss1}
&&\beta_n^\ell\Big(\eps i\beta_n E_n+\nu_E(0) E_n
\\
\nonumber&&\hspace{1cm} -\int_{0}^\infty  \nu''_E(s) \eta_{E,n}(\cdot, s)\,ds-\rot H_n\Big)=\eps  F_n\to 0\hbox{ in } J(\Omega), \\
\label{ss2}
&&\beta_n^\ell\Big(\mu i\beta_n H_n+\nu_H(0) E\\
\nonumber&&\hspace{1cm}
-\int_{0}^\infty  \nu''_H(s) \eta_{H,n}(\cdot, s)\,ds+\rot E_n\Big)= \mu  G_n\to 0\hbox{ in } \hat J(\Omega), \\
&&\beta_n^\ell\left(i\beta_n  \eta_{E,n}+\partial_s \eta_{E,n}(\cdot, s)-E_n\right)=  R_n \to 0\hbox{ in } L^2_w((0,\infty);J(\Omega)), \label{ss3}\\
&&\beta_n^\ell\left(i\beta_n \eta_{H,n}+\partial_s\eta_{H,n}(\cdot, s)-H_n\right)= S_n
\to 0\hbox{ in } L^2_w((0,\infty); \hat J(\Omega)). \label{ss4}
\eeq
By these two last identities, $\eta_{E,n}$ and $\eta_{H,n}$ are given by (see \eqref{max3equiv}-\eqref{max4equiv})
\beq
 \eta_{E,n}(s)=\frac{1-e^{- i\beta_n  s}}{ i\beta_n } E_n+\beta_n^{-\ell}\int_0^se^{- i\beta_n  (s-y)} R_n(y)\,dy,\label{ss5}\\
  \eta_{H,n}(s)=\frac{1-e^{- i\beta_n  s}}{ i\beta_n} H_n+\beta_n^{-\ell}\int_0^se^{- i\beta_n  (s-y)} S_n(y)\,dy.\label{ss6}
\eeq
Thus  inserting these expressions
in \eqref{ss1} and  \eqref{ss2}, we find that (compare with \eqref{max1bis}-\eqref{max2bis} and   \eqref{max1ter}-\eqref{max2ter})
\beq
\beta_n^\ell\left( i\beta_n \left(\eps + \mathcal{L} \nu_E(i\beta_n) \right) E_n-\rot H_n\right)=\eps F_n+r_n(i\beta_n),\label{ss1ter}\\
\beta_n^\ell\left(i\beta_n \left( \mu + \mathcal{L} \nu_H(i\beta_n)\right) H_n+\rot E_n\right)= \mu G_n+s_n(i\beta_n),\label{ss2ter}
\eeq
where
\beq\label{defrss}
r_n( i\beta_n )&=&\int_{0}^\infty \nu''_E(s) \int_0^se^{- i\beta_n  (s-y)} R_n(y)\,dy\,ds,\\
s_n( i\beta_n )&=&\int_{0}^\infty \nu''_H(s) \int_0^se^{- i\beta_n  (s-y)} S_n(y)\,dy\,ds,
\label{defsss}
\eeq
that have the regularity $r_n\in J(\Omega)$ and $s_n\in \hat J(\Omega)$ with
\be\label{estnormrnsn}
\|r_n( i\beta_n )\|_\Omega+\|s_n( i\beta_n )\|_\Omega\lesssim
\|R_n\|_{L^2_w((0,\infty);J(\Omega))}+\|S_n\|_{L^2_w((0,\infty);\hat J(\Omega))}=o(1).
\ee

Now multiplying \eqref{ss1ter} (resp. \eqref{ss2ter}) by $\bar E_n$ (resp. $\bar H_n$),
  integrating in $\Omega$, and summing the two identities we get
  \beqs
&&\beta_n^\ell  \int_\Omega
  \left(\left( i\beta_n \left(\eps + \mathcal{L} \nu_E(i\beta_n) \right) E_n-\rot H_n\right)\cdot \bar E_n
  \right.
  \\
&& \hspace{2cm}+ \left.
   \left(i\beta_n \left( \mu + \mathcal{L} \nu_H(i\beta_n)\right) H_n+\rot E_n\right)\cdot \bar H_n
  \right)\,dx
  \\
 && \hspace{2cm}=\int_\Omega\left( (\eps F_n+r_n(i\beta_n))\cdot \bar E_n+
  ( \mu G_n+s_n(i\beta_n))\cdot \bar H_n\right)\,dx.
  \eeqs
Again by Green's formula  \eqref{green}, and taking the real part, we find
 \beqs
&&\beta_n^\ell   \Re \int_\Omega
   \left(  i\beta_n \left(\eps + \mathcal{L} \nu_E(i\beta_n) \right)  |E_n|^2
 +
   i\beta_n \left( \mu + \mathcal{L} \nu_H(i\beta_n)\right) |H_n|^2
  \right)\,dx
  \\
 && \hspace{2cm}= \Re  \int_\Omega\left( (\eps F_n+r_n(i\beta_n))\cdot \bar E_n+
  ( \mu G_n+s_n(i\beta_n))\cdot \bar H_n\right)\,dx.
  \eeqs
 Owing to \eqref{bdsequence}, \eqref{ss1}, \eqref{ss2}, and \eqref{estnormrnsn}, this right-hand side tends to zero as $n$ goes to infinity, in other words, we have
 \be\label{ss10}
 \beta_n^\ell   \Re \int_\Omega
   \left(  i\beta_n \left(\eps + \mathcal{L} \nu_E(i\beta_n) \right)  |E_n|^2
 +
   i\beta_n \left( \mu + \mathcal{L} \nu_H(i\beta_n)\right) |H_n|^2
  \right)\,dx=o(1).
  \ee

 Taking into account our assumption \eqref{passitivityassump++},  for $n$ large enough, the previous property implies that
\[\beta_n^{\ell-m}     \int_\Omega
   \left(\sigma_E  |E_n|^2+\sigma_H |H_n|^2
  \right)\,dx=o(1).
  \]
  Hence taking $\ell=m$, we find that
\be\label{ss11}  \int_\Omega
   \left(\sigma_E  |E_n|^2+\sigma_H |H_n|^2
  \right)\,dx=o(1).
  \ee
  We then distinguish between three cases:
  \\
  1) If $\sigma_E$ and $\sigma_H$ are both positive, then \eqref{ss11} directly guarantees that
\be\label{ss12}
  \|E_n\|_\Omega+  \|H_n\|_\Omega=o(1).
  \ee
  Once this property holds, we come back to  \eqref{ss5} and \eqref{ss6} to get a contradiction with
  \eqref{bdsequence}, since we will show that
 \[
 \| \eta_{E,n} \|_{ L^2_w((0,\infty);  J(\Omega))}+
  \| \eta_{H,n}\|_{ L^2_w((0,\infty); \hat J(\Omega))}=o(1).
 \]
Let us check this property for $\eta_{E,n}$ (the treatment of $\eta_{H,n}$ is fully similar and is left to the reader), namely we will show that
\be\label{ss14}
 \| \eta_{E,n} \|_{ L^2_w((0,\infty);  J(\Omega))}\lesssim
 \| R_n \|_{ L^2_w((0,\infty);  J(\Omega))}+ \|E_n\|_\Omega
 \ee
 which by \eqref{ss3} and \eqref{ss12} leads to
 \[
 \| \eta_{E,n} \|_{ L^2_w((0,\infty);  J(\Omega))}=o(1).
 \]

The first step is to show that $ \eta_{E,n}$ belongs to $L^2_w((0,\infty);  J(\Omega))$. Indeed  the first term of the right-hand side of \eqref{ss5} clearly belongs to  $L^2_w((0,\infty);  J(\Omega))$, so let us concentrate on the second term. Namely let us set
\[
\Psi_n(s,\cdot)=\int_0^se^{- i\beta_n  (s-y)} R_n(y,\cdot)\,dy, \ \forall s
\geq 0.
\]
Then we easily see that $\Psi_n(0,\cdot)=0$
and $\Psi_n$ satisfies the transport equation
\[
\partial_s \Psi_n(s,\cdot)+i\beta_n \Psi_n(s,\cdot)=R_n(s,\cdot), \ \forall s>0.
\]
Hence multiplying this identity by $\bar \Psi_n w(s)$, and integrating in $\Omega$ and in $s\in (0,y)$ for any $y>0$, we find that
\[
\begin{array}{l}
\displaystyle{
\int_\Omega \int_0^y\big(\partial_s \Psi_n(s,x)+i\beta_n \Psi_n(s,x))\cdot\bar \Psi_n(s,x)\big)w(s)\,dsdx}
\\
\hspace{3 cm}
\displaystyle{
=
\int_\Omega \int_0^y R_n(s,x) \cdot\bar \Psi_n(s,x)w(s)\,dsdx.}
\end{array}
\]
Taking the real part of this identity, we find
\[\frac{1}{2}
\int_\Omega \int_0^y \partial_s (|\Psi_n(s,x)|^2)w(s)\,dsdx=
\Re\int_\Omega \int_0^y R_n(s,x) \cdot\bar \Psi_n(s,x) w(s)\,dsdx.
\]
By an integration by parts in this left-hand side, we obtain
\[
\begin{array}{l}
\displaystyle{
\delta \int_\Omega \int_0^y |\Psi_n(s,x)|^2w(s)\,dsdx+ \int_\Omega   |\Psi_n(s,y)|^2w(y)}\\
\hspace{3 cm}=\displaystyle{
2\Re\int_\Omega \int_0^y R_n(s,x) \cdot\bar \Psi_n(s,x) w(s)\,dsdx.}
\end{array}
\]
Hence Cauchy-Scwharz's inequality leads to
\[
\delta \big(\int_\Omega \int_0^y |\Psi_n(s,x)|^2w(s)\,dsdx \big)^{\frac{1}{2}}\leq
2 \big(\int_\Omega \int_0^y |R_n(s,x)|^2 w(s)\,dsdx\big)^{\frac{1}{2}}.
\]
Passing to the limit in $y$ tending to infinity we deduce that
$\Psi_n\in L^2_w((0,\infty);  J(\Omega))$ with
\be\label{ss13}
 \|\Psi_n\|_{ L^2_w((0,\infty);  J(\Omega))}
 \lesssim
 \| R_n \|_{ L^2_w((0,\infty);  J(\Omega))}.
 \ee
Coming back to  \eqref{ss5}, we then have
\beq\label{eq:serge:11:09}
 \| \eta_{E,n}\|_{ L^2_w((0,\infty);  J(\Omega))}
 &\leq& \|\frac{1-e^{- i\beta_n  \cdot }}{ i\beta_n } E_n\|_{ L^2_w((0,\infty);  J(\Omega))}
 \\
& +&\beta_n^{-\ell}\|\Psi_n\|_{ L^2_w((0,\infty);  J(\Omega))}.
\nonumber
 \eeq
 Let us then estimate the
first term of this right-hand side. First
  we notice that
 \beqs
 \|\frac{1-e^{- i\beta_n  \cdot }}{ i\beta_n } E_n\|_{ L^2_w((0,\infty);  J(\Omega))}^2
 &=&(\int_\Omega |E_n(x)|^2\,dx)\left(\int_0^\infty |\frac{1-e^{- i\beta_n  s}}{ i\beta_n }|^2 w(s)\,ds\right)
 \\
 &\leq& \frac{4}{\beta_n^2 } \left(\int_0^\infty w(s)\,ds \right) \|E_n\|_\Omega^2.
 \eeqs

 Hence for $n$ large enough, we have
 \[
  \|\frac{1-e^{- i\beta_n  \cdot }}{ i\beta_n } E_n\|_{ L^2_w((0,\infty);  J(\Omega))}
  \lesssim \|E_n\|_\Omega.
  \]
 Using this estimate and \eqref{ss13} into \eqref{eq:serge:11:09} leads to \eqref{ss14}.
  \\
  2) If $\sigma_E$ is positive, then \eqref{ss11} only yields
\be\label{ss20}
  \|E_n\|_\Omega =o(1).
  \ee
  Hence, to obtain a contradiction,  it remains to show that
\be\label{ss21}
 \|H_n\|_\Omega=o(1).
  \ee
To do so, we first multiply \eqref{ss2ter}  by $\bar H_n$ and integrate in $\Omega$
to get
\[
 \left( \mu + \mathcal{L} \nu_H(i\beta_n)\right) \int_\Omega |H_n|^2\,dx
 +\frac{1}{i\beta_n}\int_\Omega \rot E_n\cdot \bar H_n\,dx=o(1).
\]
Then using Green's formula \eqref{green}, we get
\[
 \left( \mu + \mathcal{L} \nu_H(i\beta_n)\right) \int_\Omega |H_n|^2\,dx
 +\frac{1}{i\beta_n}\int_\Omega  E_n\cdot \rot  \bar H_n\,dx=o(1).
\]
Now we use \eqref{ss1ter} to get
\be\label{ss22}
 \left( \mu + \mathcal{L} \nu_H(i\beta_n)\right) \int_\Omega |H_n|^2\,dx
 -\left(\eps + \overline{\mathcal{L} \nu_E(i\beta_n)} \right)  \int_\Omega  |E_n|^2 \,dx=o(1).
\ee
But we notice that
\eqref{LapT10} guarantees that
\[
| \mathcal{L} \nu_E(i\beta_n)|+| \mathcal{L} \nu_H(i\beta_n)|=o(1).
\]
This property combined with  \eqref{bdsequence} allows to transform \eqref{ss22}
into
\[
  \mu   \int_\Omega |H_n|^2\,dx
 - \eps   \int_\Omega  |E_n|^2 \,dx=o(1).
\]
Therefore \eqref{ss21} holds owing to \eqref{ss20}.
\\
3) If $\sigma_H$ is positive, then \eqref{ss11} only yields \eqref{ss21}
but the previous argument shows that then \eqref{ss20} holds.

The proof is then complete.
\end{proof}

\br
{\rm
The estimate \eqref{ss13} is in accordance with \eqref{estimeetermetransport} because this last one
combined with
Lemma \ref{lemraoexp}
shows that the resolvent of the transport operator is bounded (in the $L^2_w$-norm) in the imaginary axis.
}
\er

This Lemma and Lemma \ref{lemraoexp} (resp. \ref{lemrao}) directly yield the
 \bc\label{coroexpdecay}
In addition to the assumptions of Lemma   \ref{liso}, assume that  \eqref{passitivityassump++} holds with $m=0$. Then the semigroup $(e^{t{\mathcal A}})_{t\ge 0}$ is   exponentially stable, in particular
the solution $(E(t), H(t))$ of \eqref{Maxintdif}, \eqref{electricbc}
 and  \eqref{IC} tends exponentially to zero in $\caH_1$.
\ec

 \bc\label{coroplodecay}
In addition to the assumptions of Lemma   \ref{liso}, assume that  \eqref{passitivityassump++} holds with  $m>0$. Then the semigroup $(e^{t{\mathcal A}})_{t\ge 0}$  is  polynomially stable, i.e.
$$\|e^{t{\mathcal L}}U_0\| \lesssim {t^{-\frac{1}{m}}} \|U_0\|_{{\mathcal D}({\mathcal A})},\quad
 \forall U_0\in\mathcal{D}(\mathcal{A}),\quad\forall t>1.$$
In particular
the solution $(E(t), H(t))$ of \eqref{Maxintdif}, \eqref{electricbc}
 and  \eqref{IC} satisfies
 \[
 \|(E(t), H(t)\|_{\caH_1}\lesssim {t^{-\frac{1}{m}}} \|(E_0, H_0)\|_{{\mathcal D}({\mathcal B})},\quad
 \forall (E_0, H_0)\in\mathcal{D}(\mathcal{B}),\quad\forall t>1.\]
\ec


\section{Some illustrative examples\label{sexamples}}

\subsection{Some dispersive models}
All physical examples of dispersive models
 that we found in the literature (see \cite{Jackson}, \cite{MR1621212}, \cite[\S 11.2]{Sihlova},
 \cite{Cassieretall:17}, \cite{Becacheetall:18}, and \cite{Nguyen-Vinoles:18}) enter in the following example.

 Let $J$ be a positive integer and for all $j\in \{1,\cdots, J\}$, let $p_j, q_j$ be real-valued polynomial (of one variable). Let $z_j$ be a complex number with $\Re z_j=x_j<0$ and define
 \be\label{exo1.1}
 \nu_{E}(t)=\sum_{j=1}^J (p_j(t) \cos(y_j t)+q_j(t) \sin(y_j t))e^{x_j t},
 \ee
 where $y_j=\Im z_j$. Define similarly $\nu_H$ by taking other polynomials $p_j, q_j$  and other complex numbers
 $z_j$ with negative real parts. For simplicity we only examinate the case of $\nu_E$, when it will be necessary we will add the index $E$ or $H$.

 First it is easy to check that $\nu_{E}$ satisfies \eqref{asymptoicbehavior of derivatives}
and \eqref{hyponu''}.
Furthermore by rewritting
$\nu_E$ in the equivalent  form
\be\label{exo1.1equiv}
 \nu_{E}(t)=\sum_{j=1}^J P_j(t) e^{z_j t},
 \ee
 where $P_j$ is a (complex-valued) polynomial of degree $d_j$,
 we see that
 \[
 \mathcal{L}\nu_E(\lambda)=\sum_{j=1}^J \sum_{\ell=0}^{d_j}\frac{P_j^{(\ell)}(0)}{(\lambda-z_j)^{\ell+1}},
 \]
 where $P_j^{(\ell)}$ denotes the derivative of $P_j$ of order $\ell$. This means that $i\omega \mathcal{L}\nu_E(i\omega )$ is a rational fraction in $\omega$, more precisely
 \be\label{exo1.2}
i\omega \mathcal{L}\nu_E(i\omega )=\frac{P_r(\omega)}{Q_r(\omega)}+i \frac{P_i(\omega)}{Q_i(\omega)},
 \ee
where $P_r$, $Q_r$,  $P_i$, $Q_i$ are real-valued polynomials such that
\[
\deg P_r\leq \deg Q_r \hbox{ and } \deg P_i\leq \deg Q_i.
\]
This means that \eqref{passitivityassump} holds
if and only if
\be\label{passitivityassumpexo1}
\frac{P_{E,r}(\omega)}{Q_{E,r}(\omega)}\geq 0
\hbox{ and }
\frac{P_{H,r}(\omega)}{Q_{H,r}(\omega)}\geq 0, \ \forall \omega\in \R.
 \ee
Similarly, \eqref{passitivityassump+} is valid if and only if
$R(\omega)=\frac{P_{E,r}(\omega)}{Q_{E,r}(\omega)}+ \frac{P_{H,r}(\omega)}{Q_{H,r}(\omega)}$ satisfies
 \be\label{passitivityassump+exo1}
R(\omega)> 0, \forall \omega\in \R^*.
 \ee
By writing
 \be\label{exo1.3}
R(\omega)=\frac{\sum_{n=0}^{N_1} a_n \omega^n}{\sum_{n=0}^{N_2} b_n \omega^n},
\ee
with $N_1\leq N_2$,  $a_{N_1}\ne 0$ and $a_{N_2}\ne 0$, we notice that
two necessary conditions for \eqref{passitivityassump+exo1} are
\be\label{eq:serge:17/9}
N_2-N_1 \hbox{ even and } \frac{a_{N_1}}{b_{N_2}}>0.
\ee

Finally,  the last passivity assumption \eqref{passitivityassump++} is obviously related to the behavior at infinity
of
$
R(\omega)$. Using \eqref{exo1.3}, we deduce that
 \eqref{passitivityassump++} holds with $m=N_2-N_1$ if and only if \eqref{eq:serge:17/9} holds.

Let us finish this subsection by some particular cases.

\begin{exo}
The Debye model   \cite[\S 11.2.1]{Sihlova} corresponds to the choice
$\nu_H(t)=0$
and
$\nu_E(t)=\beta e^{-\frac{t}{\tau}}$, with $\beta$ and $\tau$ two positive real numbers.
Hence
\[
 \mathcal{L}\nu_E(\lambda)=\frac{\beta \tau}{\tau \lambda+1},
 \]
 and we find
 \[
 R(\omega)=\frac{\beta\tau^2 \omega^2 }{1+\tau^2 \omega^2}.
 \]
 This means that \eqref{passitivityassump} and \eqref{passitivityassump+}  hold and that
 \eqref{passitivityassump++} is valid with $m=0$. Hence by Corollary \ref{coroexpdecay}
 we deduce  the exponential decay of the energy
 if $\Omega$ is simply connected and its boundary connected
 (see \cite[Theorem 4.12]{nic:scl2012}, where the first  assumption is missing).
\end{exo}

\begin{exo}
The Lorentz model \cite[\S 11.2.2]{Sihlova} corresponds to the choice
$\nu_H(t)=0$
and
\[
\nu_E(t)=\beta \sin(\nu_0 t)  e^{-\frac{ \nu t}{2}},
\]
 with $\beta$, $\nu$ and $\nu_0$ three positive real numbers.
Hence
\[
 \mathcal{L}\nu_E(\lambda)=\frac{\beta \nu_0}{
 \omega_0^2+\lambda^2+\nu \lambda},
 \]
 with $\omega_0^2=\nu_0^2+\nu^2/4$.
Then we easily check that \eqref{passitivityassump} and \eqref{passitivityassump+}  hold and that
 \eqref{passitivityassump++} is valid with $m=2$. Hence by Corollary \ref{coroplodecay}
  we deduce a decay of the energy
 in $t^{-1}$
 if $\Omega$ is simply connected and its boundary connected
 (see \cite[Theorem 4.12]{nic:scl2012}, where the first  assumption is missing).
\end{exo}

\begin{exo}
The  Drude model \cite[\S 11.2.3]{Sihlova} (also called lossy Drude model) corresponds to the choice
$\nu_H(t)=0$
and
\[
\nu_E(t)=\beta (1-e^{- \nu t}),
\]
 with $\beta$ and $\nu$ two positive real numbers.
Hence
\[
 \mathcal{L}\nu_E(\lambda)=\frac{\beta \nu}{
  \nu \lambda+\lambda^2}.
 \]
Then we easily check that \eqref{passitivityassump} and \eqref{passitivityassump+}  hold and that
 \eqref{passitivityassump++} is valid with $m=2$. Again we deduce a decay of the energy
 in $t^{-1}$
 if $\Omega$ is simply connected and its boundary connected.
\end{exo}

The other examples from \cite[\S 11.2]{Sihlova} enter into our framework, we let the details to the reader.

\subsection{A more academic example}

For all $j\in \N^*$,  let $z_j$ be a complex number with $\Re z_j=x_j<0$ and let $a_j, b_j$ be real-valued numbers such that
\[
\sum_{j=1}^\infty (|a_j|+|b_j|)<\infty.
\]
Then we can
 define
 \be\label{exo2.1}
 \nu_{E}(t)=\sum_{j=1}^\infty (a_j \cos(y_j t)+b_j \sin(y_j t))e^{x_j t},
 \ee
 where $y_j=\Im z_j$. For simplicity take $\nu_H=0$.

Assuming that there exists $\xi>0$ such that
\[
x_j\leq -\xi, \ \forall j\in \N^*,
\]
then we directly check that \eqref{asymptoicbehavior of derivatives}
and \eqref{hyponu''} hold.

Furthermore by rewritting
$\nu_E$ in the equivalent  form
\be\label{exo2.1equiv}
 \nu_{E}(t)=\sum_{j=1}^\infty A_j e^{z_j t},
 \ee
 where $A_j$ is a complex number such that
\[
\sum_{j=1}^\infty |A_j|<\infty,
\]
 we see that
 \[
 \mathcal{L}\nu_E(\lambda)=\sum_{j=1}^\infty \frac{A_j}{\lambda-z_j}.
 \]
Now simple calculations show that for all $\omega\in \R^*$, we have
\beqs
\Re(i\omega  \mathcal{L}\nu_E(i\omega)&=&
\omega^2\sum_{j=1}^\infty \frac{\alpha_j}{x_j^2+(\omega-y_j)^2}
\\
&+& \omega\sum_{j=1}^\infty \frac{x_j\beta_j-y_j \alpha_j}{x_j^2+(\omega-y_j)^2},
\eeqs
when $A_j=\alpha_j+i\beta_j$, with $\alpha_j, \beta_j\in \R$.

For the sake of simplicity, we now treat two different particular cases
for which the second term of this right-hand side is zero.
\\
1. Assume that $y_j=\beta_j=0$, for all $j$; then
\[
\Re(i\omega  \mathcal{L}\nu_E(i\omega)=
\omega^2\sum_{j=1}^\infty \frac{\alpha_j}{x_j^2+\omega^2}.
\]
Hence assuming further that
\be\label{eq:serge:11/9:2}
x_j\geq -\Xi, \ \forall j\in \N,
\ee
for some positive real number $\Xi$,
we find that
\[
\xi^2+\omega^2\leq x_j^2+\omega^2\leq \Xi^2+\omega^2, \ \forall j\in \N,
\]
and consequently
\beqs
\Re(i\omega  \mathcal{L}\nu_E(i\omega)&\geq  &
\omega^2  \frac{a}{\Xi^2+\omega^2}+\frac{b}{\xi^2+\omega^2}
  \\&=  &
  \frac{\omega^2\left((a+b)\omega^2+a\xi^2+b\Xi^2\right)}{(\Xi^2+\omega^2)(\xi^2+\omega^2)},
\eeqs
where
\[
a=\sum_{j:\alpha_j>0} \alpha_j, \quad b=\sum_{j:\alpha_j<0} \alpha_j.
\]
This means  that the assumptions
\[
a+b=\sum_{j=1}^\infty \alpha_j>0
\hbox{ and } a\xi^2+b\Xi^2\geq 0
\]
guarantee that  \eqref{passitivityassump++} holds with $m=0$
and hence an exponential decay of the energy (under the same assumptions on $\Omega$
and its boundary as before).
On the contrary, if we assume that
\[
a+b=\sum_{j=1}^\infty \alpha_j= 0
\hbox{ and } a\xi^2+b\Xi^2> 0,
\]
then  \eqref{passitivityassump++} is valid with $m=2$ and again we deduce a decay of the energy
 in $t^{-1}$.
 \\
 2. Assume that    $x_j\beta_j-y_j \alpha_j=0$,  for all $j$; then
 \[
\Re(i\omega  \mathcal{L}\nu_E(i\omega)=
\omega^2\sum_{j:\alpha_j\ne 0} \frac{\alpha_j}{x_j^2+\left(\omega-\frac{x_j\beta_j}{\alpha_j}\right)^2}.
\]
As before assuming further that
\eqref{eq:serge:11/9:2} holds as well as
\[
\frac{\beta_j^2}{\alpha_j^2}\leq \Lambda, \forall j:\alpha_j\ne 0,
\]
for some positive real number $\Lambda$, one can show that there exist four positive constants
$c, C, \theta, \Theta,$ with $c<1<C,$ such that
\[
c(\theta^2+\omega^2)\leq x_j^2+\left(\omega-\frac{x_j\beta_j}{\alpha_j}\right)^2
\leq C( \Theta^2+\omega^2), \ \forall j\in \N: \alpha_j\ne 0,
\]
Therefore
\beqs
\Re(i\omega  \mathcal{L}\nu_E(i\omega)&\geq  &
 \omega^2 \frac{a}{C(\Theta^2+\omega^2)}+\frac{b}{c(\theta^2+\omega^2)}
  \\&\geq  &
  \frac{\omega^2\left((ac+bC)\omega^2+ac\theta^2+bC\Theta^2\right)}{cC(\Xi^2+\omega^2)(\xi^2+\omega^2)}.
\eeqs
Thus the assumptions
\[
ac+bC>0
\hbox{ and } ac\theta^2+bC\Theta^2\geq 0
\]
guarantee  an exponential decay of the energy, while the conditions
\[
ac+bC= 0
\hbox{ and } ac\theta^2+bC\Theta^2> 0,
\]
yield a decay of the energy
 in $t^{-1}$.

 \subsection{Another academic example}

 Take $\nu_H=0$ and
 \[
 \nu_E(t)=e^{-t^2}, \forall t\geq 0.
 \]
 Then we easily check that \eqref{asymptoicbehavior of derivatives}
and \eqref{hyponu''} hold.
Furthermore by Cauchy's theorem, one sees that
\[
i\omega\mathcal{L}(\nu_E)(i\omega)=e^{-\frac{\omega^2}{4}} (i\frac{\sqrt{\pi}}{2}\omega+ |\omega| I_{|\omega|}),\forall \omega\in \R^*,
\]
 as $\int_0^\infty e^{-t^2}\,dt=\frac{\sqrt{\pi}}{2}$ and
 \[
 I_\omega=\int_0^{\omega/2} e^{y^2} dy, \  \forall \omega>0.
 \]
 Hence
 \[
 \Re (i\omega\mathcal{L}(\nu_E)(i\omega))=e^{-\frac{\omega^2}{4}} |\omega| I_{|\omega|},
 \]
 which means that
 \eqref{passitivityassump} and \eqref{passitivityassump+}  hold. On the other hand,
 as
 \[
I_{|\omega|}\to \infty \quad \hbox{ as }\  |\omega|\to \infty,
\]
by
 L'H\^opital's rule, we have
 \[
 \lim_{\omega\to \infty} \frac{I_\omega}{\omega^{-1}e^{\frac{\omega^2}{4}}}=
  \lim_{\omega\to \infty} \frac{1}{\frac{1}{2}-\frac{1}{\omega^2}}=2
  \]
and
we deduce that
\[
  I_{|\omega|}\sim |\omega|^{-1} e^{\frac{\omega^2}{4}}, \
\forall \ |\omega| \hbox{ large .}
\]
Hence
 for $\omega$ large enough, one deduces that
  \[
 \Re (i\omega\mathcal{L}(\nu_E)(i\omega))\gtrsim  1,
 \]
 which means that again
 \eqref{passitivityassump++} is valid with $m=0$ and by Corollary \ref{coroexpdecay}
 we deduce  the exponential decay of the energy.

\appendix
\section{Some properties of the Laplace transform\label{appendix}}

In this section, we state some results for the Laplace transform of kernels in $K$

\begin{lem} Let $\nu\in K$, then we have
\be\label{lapT2}
\nu'(t)=-\int_t^\infty \nu''(y)\,dy.
\ee
\end{lem}
\begin{proof}
First we notice that \eqref{asymptoicbehavior of derivatives} is equivalent
to the identity
\be\label{lapT1}
\nu'(0)+\int_0^\infty \nu''(y)\,dy=0,
\ee
simply because Lebesque's bounded convergence theorem guarantees that
\[
\int_0^\infty \nu''(y)\,dy=\lim_{R\to \infty}\int_0^R \nu''(y)\,dy.
\]
As $\nu$ is twice differentiable, we may write
\[
\nu'(t)=\nu'(0)+\int_0^t \nu''(y)\,dy
\]
and by \eqref{lapT1}, we get  \eqref{lapT2}.
\end{proof}

\bc
Let $\nu\in K$, then we have
\be\label{lapT3}
|\nu'(t)|\lesssim e^{-\delta t}, \ \forall t\geq 0,
\ee
as well as
\be\label{lapT4}
|\nu(t)|\lesssim 1+ e^{-\delta t}, \ \forall t\geq 0.
\ee
\ec
\begin{proof}
The estimate  \eqref{lapT3} directly follows from
 \eqref{lapT2} and the assumption \eqref{hyponu''}.

 For the second estimate we again may write
 \[
\nu(t)=\nu(0)+\int_0^t \nu'(y)\,dy,
\]
and we conclude by \eqref{lapT3}.
\end{proof}

The previous results allow to give a meaning of the Fourier-Laplace transform of $\nu\in K$
defined by
\be\label{lapTransform}
\mathcal{L} \nu(\lambda)=
 \int_{0}^\infty e^{-\lambda s} \nu(s)
\,ds,
\ee
for all $\lambda\in \C$ such that $\Re\lambda>0$.
Furthermore, the following identities will be valid
\beq\label{lapT5}
\lambda \mathcal{L} \nu(\lambda)=\nu(0)+  \mathcal{L} \nu'(\lambda),
\\
\label{lapT6}
\lambda \mathcal{L} \nu(\lambda)=\nu(0)+\frac{1}{\lambda}(\nu'(0)+\mathcal{L} \nu''(\lambda)),
\eeq
for all $\lambda\in \C$ such that $\Re\lambda>0$.

As the estimate \eqref{lapT3} guarantees that $\nu'$  is integrable, by Lebesgue's bounded convergence theorem we deduce that
\[
\mathcal{L} \nu'(\lambda)\to 0  \hbox{ as } \Re \lambda\to \infty;
\]
by \eqref{lapT5}, we then deduce that
\be\label{lapT7}
\mathcal{L} \nu(\lambda)\to 0   \hbox{ as } \Re \lambda\to \infty.
\ee

Finally, since for $\nu\in K$ its derivative is exponentially decaying at infinity (see \eqref{lapT3}), the
Fourier-Laplace transform of $\nu'$ is also well-defined on the imaginary axis
and the mapping
\[
\R\to \C: \omega\to  \mathcal{L} \nu'(i \omega)
\]
is continous and bounded. In view to \eqref{lapT5}, we then have (in the distributional sense)
\be\label{lapT8}
i \omega \mathcal{L} \nu(i \omega)=\nu(0)+  \mathcal{L} \nu'(i \omega), \forall \omega\in \R,
\ee
and consequently the mapping
\be\label{LapT10}
\omega\to i \omega  \mathcal{L} \nu(i\omega)
\hbox{ is continuous on $\R$ and bounded. }
\ee

Note also that for $\nu\in K$, and any $\omega\in \R$,  ${\mathcal L}\nu'(i\omega)$ corresponds to the Fourier transform of
 $\widetilde{\nu'}$, the extension by zero of $\nu'$ in $(-\infty, 0)$, as
 \[
\mathcal{L} \nu'(i\omega)=
 \int_{0}^\infty e^{-i\omega s} \nu'(s)
\,ds= \int_{-\infty}^\infty e^{-i\omega s} \widetilde{\nu'}(s)
\,ds.
\]

        \protect\bibliographystyle{abbrv}
    \protect\bibliographystyle{alpha}
    \bibliography{hyperbolic_delay,/Users/sergenicaise/Documents/Serge/Desktop/Biblio/control,/Users/sergenicaise/Documents/Serge/Desktop/Biblio/valein,/Users/sergenicaise/Documents/Serge/Desktop/Biblio/kunert,/Users/sergenicaise/Documents/Serge/Desktop/Biblio/est,/Users/sergenicaise/Documents/Serge/Desktop/Biblio/femaj,/Users/sergenicaise/Documents/Serge/Desktop/Biblio/mgnet,/Users/sergenicaise/Documents/Serge/Desktop/Biblio/dg,/Users/sergenicaise/Documents/Serge/Desktop/Biblio/bib,/Users/sergenicaise/Documents/Serge/Desktop/Biblio/maxwell,/Users/sergenicaise/Documents/Serge/Desktop/Biblio/bibmix,/Users/sergenicaise/Documents/Serge/Desktop/Biblio/cochez,/Users/sergenicaise/Documents/Serge/Desktop/Biblio/soualem,/Users/sergenicaise/Documents/Serge/Desktop/Biblio/nic}

\begin{thebibliography}{10}

\bibitem{AmroucheBernardiDaugeGirault98}
C.~Amrouche, C.~Bernardi, M.~Dauge, and V.~Girault.
\newblock Vector potentials in three-dimensional non\-smooth domains.
\newblock {\em Math. Meth. Appl. Sci.}, 21:823--864, 1998.

\bibitem{arendt:88}
W.~Arendt and C.~J.~K. Batty.
\newblock Tauberian theorems and stability of one-parameter semigroups.
\newblock {\em Trans. Amer. Math. Soc.}, 306(2):837--852, 1988.

\bibitem{MR1886588}
W.~Arendt, C.~J.~K. Batty, M.~Hieber, and F.~Neubrander.
\newblock {\em Vector-valued {L}aplace transforms and {C}auchy problems},
  volume~96 of {\em Monographs in Mathematics}.
\newblock Birkh\"auser Verlag, Basel, 2001.

\bibitem{batkai:06}
A.~B{\'a}tkai, K.-J. Engel, J.~Pr{\"u}ss, and R.~Schnaubelt.
\newblock Polynomial stability of operator semigroups.
\newblock {\em Math. Nachr.}, 279(13-14):1425--1440, 2006.

\bibitem{Becacheetall:18}
E.~B\'{e}cache, P.~Joly, and V.~Vinoles.
\newblock On the analysis of perfectly matched layers for a class of dispersive
  media and application to negative index metamaterials.
\newblock {\em Math. Comp.}, 87(314):2775--2810, 2018.

\bibitem{borichev:10}
A.~Borichev and Y.~Tomilov.
\newblock Optimal polynomial decay of functions and operator semigroups.
\newblock {\em Math. Ann.}, 347(2):455--478, 2010.

\bibitem{Cassieretall:17}
M.~Cassier, P.~Joly, and M.~Kachanovska.
\newblock Mathematical models for dispersive electromagnetic waves: an
  overview.
\newblock {\em Comput. Math. Appl.}, 74(11):2792--2830, 2017.

\bibitem{ContiGattiPata08}
M.~Conti, S.~Gatti, and V.~Pata.
\newblock Uniform decay properties of linear {V}olterra integro-differential
  equations.
\newblock {\em Math. Models Methods Appl. Sci.}, 18(1):21--45, 2008.

\bibitem{DaneseGeredeliPata15}
V.~Danese, P.~G. Geredeli, and V.~Pata.
\newblock Exponential attractors for abstract equations with memory and
  applications to viscoelasticity.
\newblock {\em Discrete Contin. Dyn. Syst.}, 35(7):2881--2904, 2015.

\bibitem{GiorgiNasoPata:05}
C.~Giorgi, M.~G. Naso, and V.~Pata.
\newblock Energy decay of electromagnetic systems with memory.
\newblock {\em Math. Models Methods Appl. Sci.}, 15(10):1489--1502, 2005.

\bibitem{huang:85}
F.~L. Huang.
\newblock Characteristic conditions for exponential stability of linear
  dynamical systems in {H}ilbert spaces.
\newblock {\em Ann. Differential Equations}, 1(1):43--56, 1985.

\bibitem{Ioannidisetall_12}
A.~D. Ioannidis, G.~Kristensson, and I.~G. Stratis.
\newblock On the well-posedness of the {M}axwell system for linear
  bianisotropic media.
\newblock {\em SIAM J. Math. Anal.}, 44(4):2459--2473, 2012.

\bibitem{Jackson}
J.~D. Jackson.
\newblock {\em Classical Electrodynamics}.
\newblock John Wiley \& Sons, 1st edition, 1962.

\bibitem{Kristensson}
G.~Kristensson, A.~Karlsson, and S.~Rikte.
\newblock Electromagnetic wave propagation in dispersive and complex material
  with time domain techniques.
\newblock {\em Inverse Problems}, 1:277--294, 2002.

\bibitem{MR1621212}
G.~Kristensson, S.~Rikte, and A.~Sihvola.
\newblock Mixing formulas in the time domain.
\newblock {\em J. Opt. Soc. Amer. A}, 15(5):1411--1422, 1998.

\bibitem{Liu_Rao_05}
Z.~Liu and B.~Rao.
\newblock Characterization of polynomial decay rate for the solution of linear
  evolution equation.
\newblock {\em Z. Angew. Math. Phys.}, 56(4):630--644, 2005.

\bibitem{Lyubich-Vu:88}
Y.~I. Lyubich and Q.~P. V\~u.
\newblock Asymptotic stability of linear differential equations in {B}anach
  spaces.
\newblock {\em Studia Math.}, 88(1):37--42, 1988.

\bibitem{Munoz-Naso-Vuk:04}
J.~E. Mu\~{n}oz Rivera, M.~G. Naso, and E.~Vuk.
\newblock Asymptotic behaviour of the energy for electromagnetic systems with
  memory.
\newblock {\em Math. Methods Appl. Sci.}, 27(7):819--841, 2004.

\bibitem{Nguyen-Vinoles:18}
H.-M. Nguyen and V.~Vinoles.
\newblock Electromagnetic wave propagation in media consisting of dispersive
  metamaterials.
\newblock {\em C. R. Math. Acad. Sci. Paris}, 356(7):757--775, 2018.

\bibitem{nic:00}
S.~Nicaise.
\newblock Exact boundary controllability of {M}axwell's equations in
  heterogeneous media and an application to an inverse source problem.
\newblock {\em SIAM J. Control Optim.}, 38(4):1145--1170 (electronic), 2000.

\bibitem{nic:scl2012}
S.~Nicaise.
\newblock Stabilization and asymptotic behavior of dispersive medium models.
\newblock {\em Systems Control Lett.}, 61(5):638--648, 2012.

\bibitem{pruss:84}
J.~Pr{\"u}ss.
\newblock On the spectrum of {$C_{0}$}-semigroups.
\newblock {\em Trans. Amer. Math. Soc.}, 284(2):847--857, 1984.

\bibitem{Sihlova}
A.~Sihvola.
\newblock {\em Electromagnetic mixing formulas and applications}.
\newblock Number~47 in Electromagnetic waves series. IET, 1999.

\end{thebibliography}

\end{document}